\documentclass[preprint,review,3p,times,11pt]{elsarticle}

\usepackage{txfonts}

\usepackage{stmaryrd}
\usepackage{setspace} 
\usepackage{booktabs}
\usepackage{bbm}
\usepackage{amsthm,amssymb,amstext,amscd,amsfonts,amsbsy,amsxtra,latexsym,amsmath}
 \usepackage{mathtools}
\usepackage[english]{babel}
\usepackage[latin1]{inputenc}
\usepackage{verbatim}
\usepackage{mathrsfs}
\usepackage[all]{xy}
\usepackage{enumerate}

\usepackage[T1]{fontenc}

\biboptions{square,numbers,sort&compress} 

\newcommand{\field}[1]{\mathbb{#1}}

\newcommand{\Z}{\field{Z}}

\newcommand{\C}{\field{C}}
\newcommand{\Q}{\field{Q}}

\newcommand{\sgn}{\operatorname{sgn}}

\newcommand{\SL}{\operatorname{SL}}

\newcommand{\im}{\text{Im}}

\newcommand{\wt}{\kappa}
\newcommand{\e}{e}

\renewcommand{\epsilon}{\varepsilon}
\renewcommand{\theta}{\vartheta}
\renewcommand{\rho}{\varrho}
\renewcommand{\phi}{\varphi}

\numberwithin{equation}{section}
\newtheorem{theorem}{\textbf{Theorem}}
\numberwithin{theorem}{section}

\numberwithin{conjecture}{section}
\newtheorem{lemma}[theorem]{\textbf{Lemma}}
\newtheorem{proposition}[theorem]{\textbf{Proposition}}
\newtheorem{corollary}[theorem]{\textbf{Corollary}}

\theoremstyle{remark}
\newtheorem{remark}[theorem]{\bf Remark}
\newtheorem*{remark*}{\bf Remark}
\newtheorem*{remarks*}{\bf Remarks}

\newtheorem*{question*}{{\bf Question}}

\newtheorem*{example*}{\bf Example}

\newcommand{\smat}[4]{\bigl(\begin{smallmatrix} #1 &  #2 \\ #3 & #4 \end{smallmatrix} \bigr)}

\newcommand{\call}{\mathcal}

\newcommand{\rl}[1]{\left( #1 \right)}

\newcommand*\diff{\mathop{}\!\mathrm{d}}

\newcommand{\cQ}{\call{Q}}

\journal{ }

\begin{document}

\begin{frontmatter}



\title{\textsc{A family of mock theta functions of weights $1/2$ and $3/2$ and their congruence properties}}


\author{\textsc{Ren\'e Olivetto}}

\address{Mathematical Institute \\
University of Cologne \\
Gyrhofstr. 8b,  
50931 Cologne (Germany)}

 \ead{rolivett@math.uni-koeln.de}

\begin{abstract}
In a private communication, K. Ono conjectured that any  mock theta function of weight 1/2 or 3/2 can be congruent modulo a prime $p$ to a weakly holomorphic modular form for just a few values of $p$. In this paper we describe when such a congruence occurs. More precisely we show that it depends on the $p$-adic valuation of the mock theta function itself. In order to do so, we construct a family of mock theta functions in terms of derivatives of the Appell sum, which have a special Fourier expansion at infinity.
\end{abstract}


\end{frontmatter}


\section{\textsc{Introduction and Statement of  results}}
\label{Intro}

Basic properties of the $q$-Pochhammer symbol $(x;q)_n:=(1-x)(1-xq)(1-xq^2)\cdots(1-xq^{n-1})$ imply that 
\begin{equation} \label{eq:RamCong}
f(q):=1+\sum_{n\geq 1}\frac{q^{n^2}}{(-q;q)_n^2} \equiv 1+\sum_{n\geq 1}\frac{q^{n^2}}{(q;q)_n^2}=:P(q) \pmod{4},
\end{equation}
where a congruence between two power series is meant coefficient-wise (see Section \ref{Cong} for a precise definition). 
The right hand side of \eqref{eq:RamCong} coincides with
the generating function for the partition function $p(n)$, i.e.,
\[
P(q)=1+\sum_{n\geq 1} p(n)q^n.
\]
We recall that a \emph{partition} of a non-negative integer $n$ is any non-increasing sequence of positive integers whose sum is $n$, and, as usual, $p(n)$ denotes the number of partitions of $n$. 
Some of the most interesting arithmetic properties of $p(n)$ are Ramanujan's linear congruences 
\begin{align*}
p(5n+4)& \equiv 0 \pmod{5}, \\
p(7n+5)& \equiv 0 \pmod{7}, \\
p(11n+6)& \equiv 0 \pmod{11},
\end{align*}
and their prime power extensions proved by A. O. L. Atkin \cite{At1,At2}, S. Ramanujan \cite{R}, and G. N. Watson \cite{W}.
More recently,  S. Ahlgren and K. Ono \cite{AO1,A,O3} proved that linear congruences for p(n) exist for all moduli coprime to $6$. For further examples, see \cite{O2}, and the references therein. All these results rely on the modular transformation properties of the generating function $P(q)$.
Indeed, it is well known that $P(q)$ coincide (up to $q$-powers) with the reciprocal of the Dedekind eta-function,
\[
\eta(\tau):=q^{\frac{1}{24}}\prod_{n\geq 1} (1-q^n) , \quad \quad q:=e^{2\pi i \tau}.
\]
In particular, the function $P$ is a \emph{weakly holomorphic modular form}, i.e., a meromorphic modular form whose poles (if any) are supported at cusps. 
Using similar methods of S. Ahlgren and K. Ono, S. Treneer generalized their result to any weakly holomorphic modular form \cite{Tr}.

Coming back to \eqref{eq:RamCong},  we consider now the left hand side. 
In his last letter to Hardy dated 1920, Ramanujan listed $17$ hypergeometric series, which he called \emph{mock theta functions}, describing some of their basic properties, but without giving any precise definition. The function $f(q)$, i.e., the left hand side of \eqref{eq:RamCong}, is one of these mysterious functions.
After more than 80 years from this letter, a breakthrough was made by S. Zwegers in his 2002 Ph.D. thesis \cite{Zw}, where he characterized these special functions in three different ways, namely, in terms of Appell sums (see Subsection \ref{ss:Appell}), as Fourier coefficients of meromorphic Jacobi forms, and as quotients of indefinite binary theta series by unary theta series. For a more extensive description of mock theta functions and a survey of their  beautiful story  we refer the reader to \cite{Za,O1}.
In order to define a mock theta function, for
 any $\frac{A}{B} \in \Q$ and $\epsilon \in \{ 0,1 \}$, consider the unary theta function of weight $\frac12 + \wt \in \left\{ \frac12 ,\frac32 \right\}$ 
\[
\Theta_{\frac{A}{B},\epsilon,\wt}(\tau):=(B^*)^\wt \sum_{n\in \frac{A}{B}+\Z} (-1)^{\epsilon n}n^\wt q^{n^2},
\]
where $B^*:=B\rl{\frac12 + \left\{  \frac{B}{2} \right\}}$ and $\{ x \}$ denotes the fractional part of $x \in \Q$.
The pre image of $\Theta_{\alpha,\epsilon,\wt}$ under the operator $\xi_{\frac32 - \wt } := 2i y^{\frac32 - \wt} \overline{\frac{\partial}{\partial \overline{\tau}}}$
 yields the non-holomorphic theta function
\begin{equation} \label{eq:100}
R_{\alpha,\epsilon,\wt}(\tau):= (B^*)^{1-\wt}  \sum_{n\in \alpha +\Z} \frac{(-1)^{\epsilon n}n^\wt }{|n|^{2\wt -1}} \beta_\wt \rl{4n^2y} q^{-n^2},
\end{equation}
where $y=\im (\tau)$ and $\beta_\wt (t):=\int_t^\infty u^{\wt -\frac32} \e^{-\pi u } \diff u$. 
Following Zagier \cite{Za}, we define  a \emph{mock theta function} of weight $\frac32-\wt$ as a $q$-series $H(q)=\sum_{n\geq 0} a(n)q^n$ such that there exists a rational number $\lambda$ and a unary theta function $\Theta_{\frac{A}{B},\epsilon,\wt}$, such that
$q^\lambda H(q)+R_{\frac{A}{B},\epsilon,\wt}(B\tau)$ is a non-holomorphic modular form of weight $\frac32-\wt$ for a congruence subgroup of $\SL_2(\Z)$. We will refer to the theta function $\Theta_{\frac{A}{B},\epsilon,\wt}$ as the \emph{shadow} of $H$. The function $f(q)$ defined above is a mock theta function of weight $\frac12$ with shadow 
$\Theta_{\frac16,0,1}$. As well as classical modular forms, the Fourier coefficients of mock theta functions often have an interesting combinatorial interpretation. 
Dyson's rank generating function, characters associated to certain Lie superalgebras, and Hurwitz' class number generating function  are  examples of mock modular forms, to mention a few. 
During the last decade results about linear congruences for mock theta functions has been studied in certain special cases 
\cite{Alf2,BrO}. Among others, we point out the remarkable result of Bringmann and Ono \cite{BO} concerning the congruence properties of Dyson's rank generating function. Such identities rely on linear relations between the non-holomorphic parts. To be more precise,  applying certain quadratic twists to  mock theta functions,
one obtains weakly holomorphic modular forms due to the cancelation of the non-holomorphic parts.
In \cite{And} N. Andersen proved that any linear congruence for the coefficients of $f(q)$
and $\omega(q)$ must come in this way. 

A natural question arises. Is Anderson's result true for any mock theta function? In other words, does the obstruction to modularity dictate an obstruction 
to congruence properties? 
In light of \eqref{eq:RamCong}, the aim of this paper is to understand whether this is just a rare example, or a more general result concerning congruences between mock theta functions and weakly holomorphic modular forms. 
\begin{remark}
As in the case of holomorphic modular forms for congruence subgroup, we identify a mock theta function with its $q$-expansion at infinity.
\end{remark}
In a private conversation, Ono  conjectured that congruences as \eqref{eq:RamCong} exist just in special cases.

\begin{question*}[Ono] \label{conj}
Suppose that  $H$ is a mock modular form with (algebraic) integer coefficients. As a function of its weight and level, can one bound the largest integer $m$ for which there is a weakly holomorphic modular form $g$ with (algebraic) integer coefficients for which 
$H\equiv g \pmod{m}$?
\end{question*} 
\begin{example*}
For Ramanujan's third order mock theta function $f$, is $m=4$?
\end{example*}
In order to answer this question, we construct a family of mock theta functions, one for each shadow $\Theta_{\alpha,\epsilon,\wt}$, 
whose Fourier coefficients at infinity have a particularly nice shape. 
We can therefore reduce to study congruence properties of these particular functions. To state the result we refer to \eqref{eq:Ek} for the definition of the weight $2$ Eisenstein series $E_2$.
\begin{theorem} \label{thm1}
Let $\wt$, $\epsilon \in \{ 0 ,1\}$ and $A$ and $B$ be positive coprime integers, then the function
\[
f_{\frac{A}{B},\epsilon,\wt}(\tau):= \frac{B^*}{\Theta_{\frac{A}{B},\epsilon,\wt}(B\tau) } \rl{\sum_{n\in \Z\setminus \{ 0 \}} \frac{(-1)^{\epsilon n} n q^{\frac{Bn^2+2An}{2}}}{1-q^{Bn}} + (-1)^{\wt} \sum_{n\in \Z} \frac{(-1)^{\epsilon n}\rl{n-\frac{2A}{B}}  q^{\frac{Bn^2-2An}{2}}   }{1-q^{Bn-2A}} +\frac{1}{12B }E_2(\tau)}
\]
is a mock theta function of weight $\frac32-\wt$ and shadow $\Theta_{\frac{A}{B},\epsilon,\wt}$. 
\end{theorem}

\begin{remarks*}
\begin{enumerate}[(i)]
\item
The function $f_{\frac{A}{B},\epsilon,\wt}\Theta_{\frac{A}{B},\epsilon,\wt}-\frac{1}{12 B}E_2(\tau)$ turn out to be the weight $2$ holomorphic projection of $\Theta_{\frac{A}{B},\epsilon,\wt}R_{\frac{A}{B},\epsilon,\wt}$. 
The holomorphic projection operator sends real analytic functions, with reasonable growth, that transform as modular forms   to (almost) holomorphic modular forms. This operator was  introduced by J. Sturm \cite{St}. 
In \cite{GZ}, B. Gross and D. Zagier show that  if the weight of the modular  transformation property is greater than $2$ then the image under the holomorphic projection is modular. If the weight is $2$ then they show that, under certain assumptions on the Fourier expansion at the cusps,  the image is modular up to the addition of a constant multiple of $E_2$.
Recently, the holomorphic projection has been used by Imamoglu-Raum-Richter \cite{IRR} in order to determine recursion formulas for the Fourier coefficients of Ramanujan's mock theta functions.

\item Unlike for weights larger than $2$, the weight $2$ holomorphic projection operator does not interchange with the slash operator. 
In particular, it is not trivial to understand the modularity property of the projection of $R_{\frac{A}{B},\epsilon,\wt}\Theta_{\frac{A}{B},\epsilon,\wt}$.
Our method gives an alternative to this issue.

\item 
The approach of Imamoglu-Raum-Richter shows that the appearance of $E_2$ in the image of the weight $2$ holomorphic projection depends on the representation associated to the transformation property of the original function. Theorem \ref{thm1} imply that a trivial irreducible  representation always appear in the decomposition into irreducible of the representation associated to mock theta functions.

\item
In Proposition \ref{pro:minimum} we will see another interesting shape for this object, which, among other applications, explains the well known Hurwitz class number relations.

\item 
Let $\eta$ be Dedekind's eta-function, and $(\eta^3)^*$ be the pre-image of $\eta^3$ under $\xi_\frac12$.
In \cite{ARZ} G. Andrews, R. Rhoades, and S. Zwegers express the holomorphic projection of $\eta \cdot (\eta^3)^*$ as a derivative of the Appell sum. 
\end{enumerate}
\end{remarks*}

Using a result of S. Treneer on congruence properties of modular forms (see Subsection \ref{ss:SerreTreneer}), we prove the following.
\begin{theorem} \label{thm2}
For any weakly holomorphic modular form $g$ of level $N$ and for any prime  $p\not| N$,
\[
f_{\frac{A}{B},\epsilon,\wt}(\tau)\Theta_{\frac{A}{B},\epsilon,\wt}(B\tau) \not\equiv g(\tau) \pmod{p}.
\]
\end{theorem}
Two mock theta functions with the same shadow differ by a weakly holomorphic modular form, therefore we expect Theorem \ref{thm2} to hold for any mock theta function. Although equation \eqref{eq:RamCong} gives a contradiction, its nature  has nothing to do with the exceptional cases excluded in Theorem \ref{thm2}, but it relies on the $p$-adic properties of $f(q)$. We describe more generally this phenomenon in the following result.


\begin{corollary} \label{cor1}
Let $H$ be a mock theta function with associated non-holomorphic part $R_{\frac{A}{B},\epsilon,\wt}$. Letting $p$ be a prime number and $j:=\nu_p(H)$ be the $p$-adic valuation of $H$, then the following is true.
\begin{enumerate}[(i)]
\item If $j<0$ then there exists a weakly holomorphic modular form $g$ of weight $\frac12$ such that $p^{-j}H(\tau)\equiv g(\tau) \pmod{p^{-j}}$. 
\item If $p^{-j}H(\tau)\equiv g(\tau) \pmod{p^\ell}$ for a weakly holomorphic modular form $g$ and an integer $\ell>0$, then either $j\leq -\ell$ or $p\mid N$, where $N$ is the level of $f_{\frac{A}{B},\epsilon,\wt}-H+pg$.
\end{enumerate}
\end{corollary}


The remainder of the paper is organized as follows. In Section 2 we recall some basic arithmetic properties of weakly holomorphic modular forms and we describe the Appell sum. In Section 3  we construct a family of mock theta functions, proving Theorem \ref{thm1}. In Section 4 we prove the arithmetic properties of the mock theta functions described in Theorem \ref{thm1}, and in Section 5 we use them  to prove Theorem \ref{thm2} and Corollary \ref{cor1}.

\section{\textsc{Preliminaries}}
\label{Pre}

In this section we recall certain arithmetic results concerning weakly holomorphic modular forms and the Eisenstein series $E_2$. Finally, we recall the definition  and we describe the transformation properties of  the Appell sum.

\subsection{\textbf{\textsc{Arithmetic properties of weakly holomorphic modular forms and $E_2$}}} \label{ss:SerreTreneer}

Arithmetic properties of weakly holomorphic modular forms have been described by Treneer \cite{Tr}. Briefly speaking, for any weakly holomorphic modular form $g$ of level $N$ and for any prime $p$ coprime with $N$, Treneer constructs a cusp form which is congruent modulo $p$ to $g$, after sieving the coefficients. 
She then uses the following result of Serre  \cite[Exercise 6.4]{S2} in order to establish congruences for $g$. To be precise, 
Serre showed that  any cusp form of integral weight  $k>1$ is annihilated modulo any prime $p$ by the $Q$th Hecke operator, for a positive proportion of the primes $Q$. This result immediately implies the following. 
\begin{proposition}[Serre] \label{pro:CuspSerre}
Suppose that 
\[
f(\tau)=\sum_{n\geq 0} c(n) q^{n} 
\]
is a cusp form of weight $k\geq 1$ and level $N$. Then for each prime $p$ a positive proportion of the primes $Q\ne p$, $Q\equiv -1 \pmod{pN}$ have the property that 
\[
c(Qn)\equiv 0 \pmod{p},
\] 
for any integer $n$.
\end{proposition}
The previously mentioned result of Treneer for weakly holomorphic modular forms follows from Proposition \ref{pro:CuspSerre}.
\begin{proposition}[Treneer] \label{pro:Treneer}
Suppose that 
\[
g(\tau)=\sum_{n\gg -\infty} c(n) q^{n} 
\]
is a weakly holomorphic modular form of weight $k\geq 1$ and level $N$. Then for each prime $p$ coprime to $N$ a positive proportion of the primes $Q\ne p$, $Q\equiv -1 \pmod{p^jN}$ have the property that for any integer $m$ sufficiently large 
\[
c(Qp^m n)\equiv 0 \pmod{p^j},
\] 
for any integer $n$ coprime to $Qp$.
\end{proposition}   
\begin{remark}
In \cite{Tr} Treneer proves Proposition \ref{pro:Treneer} for $p$ odd. For integral weight weakly holomorphic modular forms the proof can easily be shown for $p=2$ as well.
\end{remark}

We conclude by recalling the modular properties of $E_2$. It is well known that for each even $k\geq2$, the Eisenstein series  
\begin{equation} \label{eq:Ek}
E_{2k}(\tau):=1-\frac{2k}{B_{2k}}\sum_{n\geq 1}\sigma_{2k-1}(n)q^n
\end{equation}
is a modular from of weight $k$. Here, 
$\sigma_{2k-1}(n):=\sum_{d|n}d^{2k-1}$ and $B_{2k}$ the $2k$th Bernoulli number. Moreover, the Eisenstein series $E_4$ and $E_6$ freely generate the ring of modular forms on $\SL_2(\Z)$. For $k=1$ the function in \eqref{eq:Ek} is still well defined, but it does not transform as a modular from. Instead, one can easily see that its completion
\begin{equation} \label{E2Tran}
\widehat{E_2}(\tau):=E_2(\tau) -\frac{3}{\pi y}, \quad \quad (y=\im(\tau)),
\end{equation}
is a weight $2$ non-holomorphic  modular form.



\subsection{\textbf{\textsc{The Appell sum}}} \label{ss:Appell}

For $\tau \in \C$, $u \in \C \setminus \rl{\Z\tau+\Z}$, and $v\in \C$,  the Appell sum is defined by
\begin{equation} \label{eq:Appell}
\call{A}(u,v;\tau):=\e^{\pi i u }\sum_{n\in \Z} \frac{(-1)^n q^{\frac{n^2+n}{2}} \e^{2\pi inv}}{1-\e^{2\pi i u}q^n}.
\end{equation}
Zwegers used this function to construct Ramanujan's mock theta functions (to be precise, Zwegers normalized the function $\call{A}$ by the Jacobi theta funtion \eqref{eq:200}).
The heart of the first part of Zwegers' Ph.D. thesis relies  on the description of the modularity properties of $\call{A}$. In order to describe his result, we recall the definition of Jacobi's theta function
\begin{equation} \label{eq:200}
\theta(v;\tau):=\sum_{n\in \frac12+\Z} q^{\frac{n^2}{2}  }\e^{2\pi in\rl{v+\frac12}},
\end{equation}
and the non-holomorphic theta function
\[
R(u;\tau):=:=-i \sum_{n\in \frac12+\Z}\rl{   \sgn(n) +\sgn\rl{\frac{\im(u)}{\im(\tau)}+n} \rl{\beta_1 \rl{\rl{n+\frac{\im(u)}{\im(\tau)}}^2 2\im(\tau)}-1}     } q^{-\frac{n^2}{2} } \e^{-2\pi in\rl{u+\frac12}},
\]
where the function $\beta_1$ was already introduced in \eqref{eq:100}. 
Zwegers constructed the completion of the Appell sum $\call{A}$ as
\begin{equation} \label{AppComp}
\widehat{\call{A}}(u,v;\tau):=\call{A}(u,v;\tau) +\frac{i}{2} \theta(v;\tau)R(u-v;\tau).
\end{equation}
In the following proposition we recall some of the transformation properties satisfied by $\widehat{\call{A}}$.
\begin{proposition}[Zwegers \cite{Zw}] \label{pro:AppTran}
Let $\widehat{\call{A}}(u,v;\tau)$ be as above, then the following hold.
\begin{enumerate}
\item For all $\lambda$, $\mu$, $\ell$, $k \in \Z$, 
\[
\widehat{\call{A}}(u+\lambda \tau +\mu,v+\ell \tau + k;\tau)=(-1)^{k+\ell}q^{\frac{\lambda^2-2\lambda \ell}{2}} \e^{2\pi i\rl{ u(\lambda -\ell) -\lambda v   }}
\widehat{\call{A}}(u,v;\tau).
\]
\item For all $\rl{ \begin{smallmatrix}a & b \\ c &d \end{smallmatrix}  } \in \SL_2(\Z)$, 
\[
\widehat{\call{A}}\rl{\frac{u}{c\tau+d},\frac{v}{c\tau+d};\frac{a\tau +b}{c\tau+d}}=(c\tau +d )    \e^{\pi i \frac{c}{c\tau+d} \rl{2uv-u^2} } \widehat{\call{A}}(u,v;\tau).
\]
\end{enumerate} 

\end{proposition}

\section{\textsc{A family of mock theta functions}}
\label{Mock}
In this section we prove Theorem \ref{thm1}. More precisely, we construct a family of mock theta functions which has a nice expression in terms of derivative of the Appell sum
and the Eisenstein series $E_2$.

\begin{proof}[{\bf Proof of Theorem \ref{thm1}}]
The function  $f_{\frac{A}{B},\epsilon,\wt}$ is clearly holomorphic. It remains to prove that it can be completed to be modular by adding the real analytic function $R_{\frac{A}{B},\epsilon,\wt}$. Here we assume $\epsilon=0$. The computation for $\epsilon=1$ is exactly the same.
In order to do so, we rewrite $f_{\frac{A}{B},\epsilon,\wt}$ as

\begin{align} \label{eq:AppId}
(B^*)^{-1}f_{\frac{A}{B},0,\wt}(\tau) \Theta_{\frac{A}{B},0,\wt}(B\tau)
&=\partial_v \left[ \mathcal{A}(u,v;B\tau)  \right]_{   \begin{subarray}{l} \\ v=  \rl{A-B/2}\tau -1/2  \\  u=0 \end{subarray}}+(-1)^\wt \partial_v \left[ \mathcal{A}(u,v;B\tau)  \right]_{\begin{subarray}{l} \\  v= \rl{B/2-A}\tau -1/2 \\ u=(B-2A)\tau \end{subarray}} \nonumber \\
&+(-1)^\wt \frac{B-2A}{B} \mathcal{A}\rl{(B-2A)\tau,\rl{\frac{B}{2}-A}\tau -1/2 ;B\tau} +\frac{1}{12 B}E_2(\tau),
\end{align}
where $\partial_v:=\frac{1}{2\pi i}\frac{\partial}{\partial v}$.
This identity follows immediately from the definition of $\call{A}$( see \eqref{eq:Appell}).

We define the completion $\widehat{f}_{\frac{A}{B},0,\wt}(\tau)$ of $f_{\frac{A}{B},0,\wt}(\tau)$ by replacing $\mathcal{A}$ by $\widehat{\mathcal{A}}$ in \eqref{eq:AppId}. We need to show that $\widehat{f}_{\frac{A}{B},0,\wt}(\tau)$ transforms as a modular form of weight $2$ and that 
\[
\widehat{f}_{\frac{A}{B},0,\wt}(\tau)-f_{\frac{A}{B},0,\wt}(\tau) =R_{\frac{A}{B},0,\wt}(\tau).
\]
We start by proving the modularity property. To simplify the notation, for $\gamma=\smat{a}{b}{c}{d}\in \SL_2(\Z)$ we write $\gamma \tau:=\frac{a\tau+b}{c\tau+d}$. Also, we call $s:=\frac{B}{2}-A$. By definition we have
\begin{align}  \label{eq:0}
(B^*)^{-1} \widehat{f}_{\frac{A}{B},0,\wt}(\gamma\tau) \Theta_{\frac{A}{B},0,\wt}(B\cdot\gamma\tau)
&=\partial_v \left[ \widehat{\mathcal{A}}(u,v;B\cdot\gamma\tau)  \right]_{   \begin{subarray}{l} \\ v= -s\cdot \gamma\tau -1/2  \\  u=0 \end{subarray}}+(-1)^\wt \partial_v \left[ \widehat{\mathcal{A}}(u,v;B\cdot\gamma\tau)  \right]_{\begin{subarray}{l} \\  v=  s\cdot \gamma\tau -1/2 \\ u=2s\cdot \gamma\tau \end{subarray}} \nonumber \\
&+(-1)^\wt \frac{B-2A}{B} \widehat{\mathcal{A}}\rl{2s\cdot \gamma\tau,s\cdot\gamma \tau -1/2 ;B\cdot\gamma\tau} +\frac{1}{12B }E_2(\gamma\tau).
\end{align}
We study each of these four terms separately using Proposition \ref{pro:AppTran} and eq. \eqref{E2Tran}.

The first term: 
\begin{align*}
&\partial_v \left[ \widehat{\mathcal{A}}(u,v;B\cdot\gamma\tau)  \right]_{   \begin{subarray}{l} \\ v= -s\cdot\gamma\tau -1/2  \\  u=0 \end{subarray}} =
\partial_v \left[ (c\tau+d)^2 \e^{2\pi i\frac{c}{2B(c\tau+d)}(2uv-u^2)}   \widehat{\mathcal{A}}(u,v;B\tau)  \right]_{   \begin{subarray}{l} \\ v= -s(a\tau+b )-(c\tau+d)/2  \\  u=0 \end{subarray}} \\ 
&=(c\tau+d)\frac{c}{B} \lim_{u\rightarrow 0 } u  \widehat{\mathcal{A}}(u,v;B\tau) 
+(c\tau+d)^2 \partial_v \left[ \widehat{\mathcal{A}}(u,v;B\tau)  \right]_{   \begin{subarray}{l} \\ v=-s(a\tau+b )-(c\tau+d)/2  \\  u=0 \end{subarray}}.
\end{align*} 
Using  $\lim_{u\rightarrow 0 } u  \widehat{\mathcal{A}}(u,v;B\tau)=-\frac{1}{2\pi i}$ this equals
\begin{equation} \label{eq:1} 
-\frac{(c\tau+d)}{2\pi i}  \frac{c}{B} +(c\tau+d)^2 \partial_v \left[ \widehat{\mathcal{A}}(u,v;B\tau)  \right]_{   \begin{subarray}{l} \\ v=  -s\tau -1/2  \\  u=0 \end{subarray}}.
\end{equation}
The second term:
\begin{align*}
&\partial_v \left[ \widehat{\mathcal{A}}(u,v;B\cdot\gamma\tau)  \right]_{\begin{subarray}{l} \\  v= s\cdot\gamma\tau -1/2 \\ u=2s\cdot \gamma\tau \end{subarray}}=
(c\tau+d)^2 \partial_v \left[  \e^{2\pi i\frac{c}{2B(c\tau+d)}(2uv-u^2)}   \widehat{\mathcal{A}}(u,v;B\tau)  \right]_{   \begin{subarray}{l} \\ v= s(a\tau+b )-(c\tau+d)/2  \\  u=2s(a\tau+b) \end{subarray}} \\
&=(c\tau+d)^2  q^{-\frac{sac}{B}} \rl{             
 \partial_v \left[ 
 \widehat{\mathcal{A}}\rl{u+2s(a-1)\tau,v+s(a-1)\tau -\frac{c\tau+d-1}{2};B\tau}
 \right]_{   \begin{subarray}{l} \\ v= s\tau-1/2  \\  u=2s\tau \end{subarray}} \right.\\
 & \left.
+\frac{2cs(a\tau+b)}{B(c\tau+d)} \widehat{\mathcal{A}}\rl{2s(a\tau+b),s(a\tau+b)-\frac{c\tau+d}{2};B\tau}}.
\end{align*}
Using the elliptic transformation properties of $ \widehat{\mathcal{A}}$ this term equals
\begin{equation}  \label{eq:2} 
(c\tau+d)^2 q^{-\frac{sac}{B}} \rl{   q^{\frac{sac}{B}}   \partial_v \left[ 
 \widehat{\mathcal{A}}\rl{u+2s\tau,v+s\tau -\frac{1}{2};B\tau}
 \right]_{   \begin{subarray}{l} \\ v= s\tau-1/2  \\  u=2s\tau \end{subarray}}     +\frac{2s}{B}\rl{1-\frac{1}{c\tau+d}} q^{\frac{sac}{B}}  \widehat{\mathcal{A}}\rl{2s\tau,s\tau-\frac{1}{2};B\tau} }.
\end{equation}
A similar computation yields the third term:
\begin{equation}  \label{eq:3} 
\widehat{\mathcal{A}}\rl{2s\cdot \gamma\tau,s\cdot\gamma \tau -1/2;B\cdot\gamma\tau} =(c\tau+d)\widehat{\mathcal{A}}\rl{2s\tau,s\tau -1/2;B\tau} .
\end{equation}
Finally, using \eqref{E2Tran} we rewrite the fourth term as
\begin{equation}  \label{eq:4} 
E_2(\gamma\tau) =(c\tau+d)^2 E_2(\tau) +\frac{6c(c\tau+d)}{\pi i}.
\end{equation}
From \eqref{eq:1}, \eqref{eq:2}, \eqref{eq:3}, and \eqref{eq:4}, we rewrite \eqref{eq:0} as
\begin{multline*}
(B^*)^{-1}\widehat{f}_{\frac{A}{B},0,\wt}(\gamma\tau) \Theta_{\frac{A}{B},0,\wt}(B\cdot\gamma\tau)=
(c\tau+d)^2 \partial_v \left[ \widehat{\mathcal{A}}(u,v;B\tau)  \right]_{   \begin{subarray}{l} \\ v=  -s\tau -1/2  \\  u=0 \end{subarray}}
 -\frac{(c\tau+d)}{2\pi i}  \frac{c}{B} \\
+(c\tau+d)^2 (-1)^\wt\partial_v \left[ 
 \widehat{\mathcal{A}}\rl{u+2s\tau,v+s\tau -\frac{1}{2};B\tau}
 \right]_{   \begin{subarray}{l} \\ v= s\tau-1/2  \\  u=2s\tau \end{subarray}}  \\
 +(-1)^\wt (c\tau+d)^2 \frac{2s}{B}\rl{1-\frac{1}{c\tau+d}} q^{\frac{sac}{B}}  \widehat{\mathcal{A}}\rl{2s\tau,s\tau-\frac{1}{2};B\tau}  \\
+(-1)^\wt (c\tau+d)\frac{2s}{B}\widehat{\mathcal{A}}\rl{2s\tau,s\tau -\frac12;B\tau} +\frac{(c\tau+d)^2}{12B}E_2(\tau)+\frac{6c(c\tau+d)}{12B \pi i},
\end{multline*}
which by definition of $\widehat{f}_{\frac{A}{B},0,\wt}$ equals $\widehat{f}_{\frac{A}{B},0,\wt}(\tau) \Theta_{\frac{A}{B},0,\wt}(B\tau)$.

We conclude the proof of the theorem by computing the non-holomorphic part of $\widehat{f}_{\frac{A}{B},0,\wt}$. As before, let $s:=\frac{B}{2}-A$.
From \eqref{AppComp} and \eqref{eq:0} we have
\begin{multline*}
\widehat{f}_{\frac{A}{B},0,\wt}(\tau) \Theta_{\frac{A}{B},0,\wt}(B\tau) =
f_{\frac{A}{B},0,\wt}(\tau) \Theta_{\frac{A}{B},0,\wt}(B\tau)
+\frac{iB^*}{2}\rl{        \theta^\prime\rl{-s\tau-\frac12;B\tau}R\rl{s\tau+\frac12;B\tau} \right. \\ \left.
+ \theta\rl{-s\tau-\frac12;B\tau}R^\prime\rl{s\tau+\frac12;B\tau} 
+(-1)^\wt  \theta^\prime\rl{s\tau-\frac12;B\tau}R\rl{s\tau+\frac12;B\tau}  \right. \\ \left.
+(-1)^\wt \theta\rl{s\tau-\frac12;B\tau}R^\prime\rl{s\tau+\frac12;B\tau} 
+(-1)^\wt \frac{2s}{B}\theta\rl{s\tau-\frac12;B\tau}R\rl{s\tau+\frac12;B\tau} 
     },
\end{multline*}
where $\theta^\prime$ and $R^\prime$ denote the derivatives in the elliptic variable.
The result follows using the parity properties of $\theta$ and $R$.
\end{proof}

Note that  the Fourier expansion   of the mock modular form 
$f_{\frac{A}{B},\epsilon,\wt}$ at any cusp has integral coefficients and its growth at the cusps is dictated by the decay of $\Theta_{\frac{A}{B},\epsilon,\wt}$. 
In fact, the Fourier expansion at infinity of $f_{\frac{A}{B},\epsilon,\wt}$ has a particular and nice shape, which we describe in the following proposition.
\begin{proposition} \label{pro:minimum}
Let $f_{\frac{A}{B},\epsilon,\wt}\Theta_{\frac{A}{B},\epsilon,\wt}$ be as above, then 
\[
f_{\frac{A}{B},\epsilon,\wt}(\tau)\Theta_{\frac{A}{B},\epsilon,\wt}(B\tau)=\frac{B^*}{12B}+ \frac{B^*}{B}\sum_{n>0}\rl{ \sum_{(a,b)\in V_n } \sgn\rl{b^2-B^2a^2}^\wt (-1)^{\epsilon a} \min\left\{ |b|,|Ba|  \right\} +2\sigma\rl{\frac{n}{2}}}q^{\frac{n}{2}},
\]
where $V_n:=\left\{ (a,b)\in \Z^2 \colon \substack{ab=n,\\ \quad b+Ba\equiv 2A \pmod{2B}  } \right\}$
\end{proposition}
\begin{proof}[{\bf Proof}] 
We refer to the statement of Theorem \ref{thm1} for the definition of $f_{\frac{A}{B},\epsilon,\wt}\Theta_{\frac{A}{B},\epsilon,\wt}$. We consider each of the three summands separately. Expanding the denominator in the first summand, which we call here $\Sigma_1$, in a geometric series, it equals
\begin{equation} \label{eq:11}
\Sigma_1=\sum_{\substack{n>0 \\ m\geq 0}} (-1)^{\epsilon n} n q^{\frac{Bn^2+2An}{2}} q^{Bmn} -
\sum_{\substack{n<0 \\ m< 0}} (-1)^{\epsilon n} n q^{\frac{Bn^2+2An}{2}} q^{Bmn}.
\end{equation}
Replacing $n$ by $s-m$, we rewrite \eqref{eq:11} as 
\[
\Sigma_1=\rl{\sum_{s>m\geq 0} -\sum_{s<m<0} }(-1)^{\epsilon (s-m)} (s-m) q^{\frac{s-m}{2}\rl{2A+B(s+m)}} .
\]
Similarly, the second summand  in the definition of $f_{\frac{A}{B},\epsilon,\wt}\Theta_{\frac{A}{B},\epsilon,\wt}$, which we call $\Sigma_2$, can be written as
\[
\Sigma_2=(-1)^{\wt}  \rl{  \sum_{-s\leq m <0} - \sum_{-s>m\geq 0}  }(-1)^{\epsilon (s-m)} \rl{s+m +\frac{2A}{B} } q^{\frac{s-m}{2}\rl{2A+B(s+m)}}.
\]
In particular we can write $\Sigma_1+\Sigma_2$ in a unique formula as
\begin{equation} \label{eq:12}
\Sigma_1+\Sigma_2=\sum_{\substack{s,m\in \Z \\ |s|>|m|}} \sgn\rl{sr}^\wt (-1)^{\epsilon (s-m)} \rl{\left| s+\frac{A}{B} \right|-\left| m+\frac{A}{B} \right|} q^{\frac{s-m}{2}\rl{2A+B(s+m)}}.
\end{equation}
We fix a $q$-exponent $n$. From \eqref{eq:12} we know that $n=(s-m)\rl{2A+B(s+m)}$ for certain integers $s$ and $m$. Since the map 
\begin{align*}
U_n:=\left\{  (s,m)\in \Z^2 \colon   n=\rl{s-m}\rl{2A+B(s+m)} \right\}        & \rightarrow       
 V_n:=    \left\{ (a,b)\in \Z^2 \colon \substack{ab=n,\\ \quad b+Ba\equiv 2A \pmod{2B}  } \right\}  \\
(s,m) \quad \quad \quad \quad &\mapsto \quad \quad \quad  (a,b)=(s-m,2A+B(s+m))
\end{align*}
is a bijection, equation \eqref{eq:12} equals
\begin{multline*}
\frac{B^*}{2B} \sum_{n>0}\rl{ \sum_{(a,b)\in V_n} \sgn\rl{b^2-B^2a^2}^\wt (-1)^{\epsilon a} \rl{\left| b+Ba \right|-\left| b-Ba \right|} }q^{\frac{n}{2}} \\
=\frac{B^*}{B} \sum_{n>0}\rl{ \sum_{(a,b)\in V_n} \sgn\rl{b^2-B^2a^2}^\wt (-1)^{\epsilon a} \min\left\{ |b|,|Ba|  \right\} }q^{\frac{n}{2}}.
\end{multline*}
Adding the contribution of $E_2$ completes the proof.
\end{proof}

\section{\textsc{Congruences for $f_{\frac{A}{B},\epsilon,\wt}\Theta_{\frac{A}{B},\epsilon,\wt}$}}
\label{Cong}

In this section we prove certain congruence properties satisfied by the mixed mock modular form $f_{\frac{A}{B},\epsilon,\wt}\Theta_{\frac{A}{B},\epsilon,\wt}$.
We shall see in the next section that these properties are not satisfied by weakly holomorphic modular forms. In other words, we will see that the non-holomorphic function $R_{\frac{A}{B},\epsilon,\wt}$ cause an obstruction to the congruence between a mock modular form and a weakly holomorphic modular form. 

In order to give the precise statement, we recall the definition of $p$-adic valuation of a $q$-series. Letting $p$ be  a prime and $g(\tau)=\sum_n a(n)q^n \in \Q((q))$, then the $p$-adic valuation of $g$ is defined by
$
\nu_p(g):=\inf_n\rl{\nu_p ^*\rl{a(n)}},
$
where $\nu_p^*$ is the standard $p$-adic valuation in $\Q_p$. Moreover, two $q$-series  $g$ and $h$ are congruent modulo $p^m$
if
\[
\nu_p(g-h)\geq \nu_p(g) +m, \quad \quad \quad ( g\equiv h \pmod{p^m}).
\]
\begin{proposition} \label{pro:main}
Let $f_{\frac{A}{B},\epsilon,\wt}(\tau)\Theta_{\frac{A}{B},\epsilon,\wt}(B\tau)=\frac{B^*}{B}\sum_{n\geq 0} c(n)q^{\frac{n}{2}}$.
Then the following are true.
\begin{enumerate}
\item Let $p$ be a fixed odd prime, and $m$ be any non-negative integer. Then for  any  prime $Q\equiv -1 \pmod{p}$ sufficiently large, there exists a positive integer $k$ coprime to $Qp$ such that 
\[
c\rl{ Qp^m k  }\not\equiv 0 \pmod{p}.
\]
\item There exist infinitely many integers $m\geq2$  such that   for  any  prime $Q\equiv \pm1 \pmod{B}$ sufficiently large, there exists an integer $k$ coprime to $2Q$ such that 
\[
c\rl{ Q2^m k  }\equiv 2  \pmod{4}.
\]
\end{enumerate}
\end{proposition}

We split the proof of Proposition \ref{pro:main} according to the parity of the prime $p$. Also, we give the proof for $\epsilon =0$. The case $\epsilon=1$ 
is analogous and is proven in the author's Ph.D. thesis.

\subsection{\textsc{\textbf{Congruences modulo odd primes}}}

From Proposition \ref{pro:minimum} we know that the coefficient $c(n)$ is essentially given by the sum over the divisors $b$ of $n$ that satisfy
\begin{equation}\label{mainCong}
b+B\frac{n}{b}\equiv 2A \pmod{2B}.
\end{equation}
The contribution of $E_2$ is irrelevant since $\sigma(Qp^m n)\equiv 0 \pmod{p}$ for any integer $n$ and any prime $Q\equiv-1 \pmod{p}$.
To simplify the notation, for any fixed $A$, $B$, and $n$ we introduce the counting function $\Psi_n$  defined by
\begin{equation}
\Psi_n(b):=\begin{cases}
1, \mbox{  if $\pm b$ satisfies \eqref{mainCong},} \\
0, \mbox{ otherwise.}
\end{cases}
\end{equation}
In order to prove the odd case of Proposition \ref{pro:main}, let $\cQ:=Qp^m$ and  set
\begin{align} \label{eq:c}
c_1&:=c(4\cQ)\equiv \sum_{d|4}  \rl{(-1)^{\wt} d \Psi_{4\cQ}(d)+  B \frac{4}{d} \Psi_{4\cQ}(d\cQ)     }      \pmod{p},\nonumber\\
c_2&:=c(8\cQ)\equiv \sum_{d|8} \rl{(-1)^{\wt} d \Psi_{8\cQ}(d)+ B \frac{8}{d} \Psi_{8\cQ}(d\cQ)     }      \pmod{p},\nonumber\\
c_3&:=c(4\ell\cQ)\equiv \sum_{d|4\ell}  \rl{(-1)^{\wt} d \Psi_{4\cQ \ell}(d)+  B \frac{4\ell}{d} \Psi_{4\cQ}(d\cQ)   }      \pmod{p},\nonumber \\
c_4&:=c(8k\cQ)\equiv \sum_{d|8\ell}  \rl{(-1)^{\wt} d \Psi_{8\cQ \ell}(d)+  B \frac{8\ell}{d} \Psi_{8\cQ}(d\cQ)     }      \pmod{p},
\end{align}
where $\ell\neq p$ denotes a prime smaller enough respect to $Q$ such that $\ell\equiv A \pmod{B}$.
\begin{remark}
The congruence properties of $c_1,\cdots,c_4$ stated in \eqref{eq:c} come directly from their own definition and from the fact that $Q$ is a large prime. To be more precise, note that $c(n)$ is essentially the sum of the ``small'' divisors of $n$. In particular, consider the coefficients $c(Qp^m N)$, where $N$ is as in the definition of the $c_i$s. If $Q p^n h$ divides $Qp^m N$ (with $h|N$ and $n\geq m$), since $Q$ is large then its contribution to the value of $c(Qp^m N)$
is given by $Np^{m-n}/h$, which is $0$ modulo $p$ unless $n=m$. For this reason we will consider $\cQ=Qp^m$ as a ``prime'' factor of $N\cQ$.
\end{remark}
We split the proof of Proposition \ref{pro:main} in three cases, according to the residue classes of $B \pmod{4}$. The special case $B \equiv 0 \pmod{p}$ will be treated separately.

\subsection{Case 1: $p$ divides $B$}
Since $B\equiv 0 \pmod{p}$ we have
\begin{align*}
c_3-c_1 &\equiv (-1)^\wt \rl{ \sum_{d|4\ell}  d \Psi_{4\cQ k}(d) -\sum_{d|4}  d \Psi_{4\cQ}(d) } \\
 &\equiv (-1)^\wt \sum_{d|4}  d\ell \Psi_{4\cQ k}(d\ell)  \pmod{p}.
\end{align*}
By definition of $\ell$, $\Psi_{4\cQ \ell}(d\ell)=1$ for $d=2$ and might equal $1$ for $d=4$. If $\Psi_{4\cQ \ell}(4\ell)=0$ then 
\[
c_3-c_1\equiv (-1)^\wt 2\ell \not\equiv 0\pmod{p}.
\]  
If $\Psi_{4\cQ \ell}(4\ell)=1$ then $6A\equiv 0 \pmod{B}$, therefore $B=6$ and $p=3$. In this particular case 
\[
c_2-c_1\equiv (-1)^{\wt}4 \not\equiv 0 \pmod{3}.
\]
From now on we can assume $(p,B)=1$.

\subsubsection{Case 2: $B$ odd} 
If $B=3$ then $c_1 \equiv (-1)^\wt 2+6 \not\equiv 0 \pmod{p}$. Assuming $B$ to be odd and larger than $3$ forces $b$ and $\frac{n}{b}$ to have the same parity, which implies
\begin{align*}
c_1&\equiv (-1)^k 2\Psi_{4\cQ}(2)+2B \Psi_{4\cQ}(2\cQ) \\
c_2&\equiv (-1)^k \rl{2\Psi_{8\cQ}(2)+ 4\Psi_{8\cQ}(4) } +4B \Psi_{8\cQ}(2\cQ) +2B \Psi_{8\cQ}(4\cQ).
\end{align*}

If $\Psi_{4\cQ}(2\cQ)=1$ then $\Psi_{8\cQ}(2\cQ)=1$ and 
$\Psi_{8\cQ}(4\cQ)=0$. The first claim follows directly from the definition. 
To prove the second one, assume $\Psi_{8\cQ}(4\cQ)=1=\Psi_{8\cQ}(2\cQ)$, then 
either $6\cQ$ or $2\cQ \equiv 0 \pmod{2B}$, which is impossible since $(B,\cQ)=1$. This implies 
\begin{align*}
c_1&\equiv (-1)^\wt 2\Psi_{4\cQ}(2)+2B  \\
c_2-c_1&\equiv (-1)^\wt 4 \Psi_{8\cQ}(4) +2B.
\end{align*}
Again, $\Psi_{4\cQ}(2)+\Psi_{8\cQ}(4) \in \{ 0 ,1 \}$, therefore either $c_1$ or $c_2-c_1 \not\equiv 0 \pmod{p}$. 

If $\Psi_{4\cQ}(2\cQ)=0$ then 
\[
c_3\equiv (-1)^\wt 2\rl{ \Psi_{4\cQ}(2)+\ell  } +2B\Psi_{4\ell\cQ}(2\ell\cQ).
\]
The Chinese reminder theorem and Dirichlet's prime number theorem guarantee the existence of $\ell$ such that $c_3\not\equiv 0\pmod{p}$.

\subsubsection{Case 3: $B\equiv 0 \pmod{4}$}
Since $B \equiv 0\pmod{4}$, for each $n$, $\Psi_{n}(b)=0$ unless $b\equiv 2 \pmod{4}$. In particular, we have
\begin{align*}
c_1 & \equiv (-1)^\wt 2\Psi_{4\cQ}(2)+2B\Psi_{4\cQ}(2\cQ)     \pmod{p}, \\
c_2 & \equiv (-1)^\wt 2\Psi_{8\cQ}(2)+4B\Psi_{8\cQ}(2\cQ)     \pmod{p}.
\end{align*}
If $\Psi_{4\cQ}(2\cQ)=1$ $(=\Psi_{8\cQ}(2\cQ))$, then 
\[
c_2-c_1\equiv 2B \not\equiv 0 \pmod{p}.
\]
Otherwise, if $\Psi_{4\cQ}(2\cQ)=0$, then 
\[
c_3\equiv (-1)^\wt(2\Psi_{4\ell\cQ}(2) +2\ell)+2B\Psi_{4\ell\cQ}(2\ell\cQ).
\]
As in the previous case we can choose $\ell$ such that $c_3\not\equiv 0 \pmod{p}$.

\subsubsection{Case 4: $B\equiv 2 \pmod{4}$}
We treat  the case $B=6$ and $B=10$ separately. In the following lemma we describe the properties of $\Psi$ in this setting.
\begin{lemma} \label{lem:psiprop}
Let $\cQ >5$ be a prime number and $x\in \{ 1, \cQ \}$, then the following is true.
\begin{enumerate}
\item If  $B=6$ then
\[
\Psi_{4\cQ}(2x)=\Psi_{4\cQ}(4x)=\Psi_{8\cQ}(2x)=\Psi_{8\cQ}(8x)=1.
\]
\item If  $B=10$
 \[
 \Psi_{8\cQ}(2x)=\Psi_{8\cQ}(8x)=\Psi_{4\cQ}(2x)\neq \Psi_{4\cQ}(4x).
 \]
  \item If $B \notin \{ 6,10 \}$, $B \equiv 2 \pmod{4}$
    \begin{enumerate}
   \item $\Psi_{4\cQ}(2x)+ \Psi_{4\cQ}(4x) \in \{ 0,1 \}$.
   \item $\Psi_{8\cQ}(2x)+\Psi_{8x}(8x) \in \{0,1\}$.
   \item $\Psi_{4\cQ}(4x)+\Psi_{8x}(8x) \in \{0,1\}$.
   \end{enumerate}
\end{enumerate}
\end{lemma}
If $B=6$ Lemma \ref{lem:psiprop} implies 
\[
c_1\equiv (-1)^\wt 6 +18 \not\equiv 0\pmod{p}.
\]

Assume $B=10$. If $\Psi_{4\cQ}(2)=1$ then by Lemma \ref{lem:psiprop} 
\begin{align*}
c_1&\equiv (-1)^\wt 2 +10(2\Psi_{4\cQ}(2\cQ)+\Psi_{4\cQ}(4\cQ)) \\
c_2&\equiv (-1)^\wt 10 +10(4\Psi_{8\cQ}(2\cQ)+\Psi_{8\cQ}(8\cQ)).
\end{align*}
If $\Psi_{4\cQ}(2\cQ)=1$ then $c_2\not\equiv 0\pmod{p} $ unless $p=3$ and $\wt=0$, in which case $c_1\equiv 22 \not\equiv 0\pmod{p}$.
On the other hand, if $\Psi_{4\cQ}(2\cQ)=0$ then $c_2 \equiv \pm 10\not\equiv 0\pmod{p}$, since $(p,10)=1$.
Finally, if $\Psi_{4\cQ}(2)=0$ then
\begin{align*}
c_1&\equiv (-1)^\wt 4 +10(2\Psi_{4\cQ}(2\cQ)+\Psi_{4\cQ}(4\cQ)) \\
c_2&\equiv 10(4\Psi_{8\cQ}(2\cQ)+\Psi_{8\cQ}(8\cQ)).
\end{align*}
If $\Psi_{4\cQ}(2\cQ)=1$ then $c_2\equiv 50\not\equiv 0 \pmod{p}$. Otherwise, using the same argument as in the previous cases, it is possible to find a prime $\ell$ such that either $c_1$
or $c_3\not\equiv 0 \pmod{p}$.

Finally, assume $B\notin \{ 6,10\}$. Since $B\equiv 2 \pmod{4}$ then $\Psi_n(b)=0$ unless $b$ is even and $\frac{b}{2}$ and $\frac{n}{b}$ have opposite parity. As a consequence, we have
\begin{align*}
c_1 & \equiv (-1)^\wt \rl{ 2\Psi_{4\cQ}(2)+ 4\Psi_{4\cQ}(4)    }+B\rl{2 \Psi_{4\cQ}(2\cQ)  +  \Psi_{4\cQ}(4\cQ)}    \pmod{p}, \\
c_2 & \equiv (-1)^\wt \rl{ 2\Psi_{8\cQ}(2)+ 8\Psi_{8\cQ}(8)    }+B\rl{4 \Psi_{8\cQ}(2\cQ)  +  \Psi_{8\cQ}(8\cQ)}   \pmod{p}.
\end{align*}
If $\Psi_{4\cQ}(2\cQ)=1$ then 
\begin{align*}
c_1 & \equiv (-1)^\wt \rl{ 2\Psi_{4\cQ}(2)+ 4\Psi_{4\cQ}(4)    }+2B   \pmod{p}, \\
c_2 & \equiv (-1)^\wt \rl{ 2\Psi_{8\cQ}(2)+8\Psi_{8\cQ}(8)    }+4B   \pmod{p},
\end{align*}
therefore either $c_1$, or $c_2$, or $c_2-c_1 \not\equiv 0 \pmod{p}$.

If $\Psi_{4\cQ}(2\cQ)=0$ then as before we can determine $\ell$ such that $c_3\not\equiv 0 \pmod{p}$.

\subsection{\textbf{\textsc{ Congruences modulo $2$}}}

The proof of Proposition \ref{pro:main} for $p=2$ is analogous to the proof in the case of odd $p$, therefore we do not prove it here. We only mention that in this case one can show that a linear combination of 
\begin{align*}
c_1 &:= 2c\rl{ Q2^m } \\
c_2 &:=2 c\rl{ Q2^m \ell  }
\end{align*}
is congruent to $2 \pmod{4}$. A detailed proof can be found in the author's  Ph.D. thesis.

\section{\textsc{Proof of the main results}}
\label{Proof}

In this section we prove Theorem \ref{thm2} and Corollary \ref{cor1}. 

\begin{proof}[{\bf Proof of Theorem \ref{thm2}}] Assume by contradiction that there exists a weakly holomorphic modular form $g$ and a prime $p$ such that $f_{\frac{A}{B},\epsilon,\wt}\Theta_{\frac{A}{B},\epsilon,\wt} \equiv g \pmod{p}$. 
Proposition \ref{pro:Treneer} applied to $g$ implies that there exists infinitely many primes $Q$ such that 
\[
c(Qp^m n)\equiv 0 \pmod{p}
\]
for any  integer $n$ coprime to $Qp$. This contradict Proposition \ref{pro:main}.  
\end{proof}

\begin{proof}[{\bf Proof of Corollary \ref{cor1}}]
Since $H$ has the same non-holomorphic part as $f_{\frac{A}{B},\epsilon,\wt}$, then the difference $m:=H-f_{\frac{A}{B},\epsilon,\wt}$ is a weakly holomorphic modular form of weight $\frac12$ and level $N$. 
\begin{enumerate}[(i)]
\item Since $\nu_p(H)=j<0$, then $\nu_p(m)=\nu_p(H-f_{\frac{A}{B},\epsilon,\wt})=\nu_p(H)$. In particular $\nu_p(p^{-j}m)=0$. Therefore, 
\[
\nu_p(p^{-j}H-p^{-j}m)=\nu_p(p^{-j}f_{\frac{A}{B},\epsilon,\wt})=-j=\nu_p(p^{-j}m)-j, 
\]
in other words $p^{-j}H\equiv p^{-j}m \pmod{p^{-j}}$.
\item Conversely, assume that there exists a weakly holomorphic modular form $g$   such that 
\begin{equation}\label{eq:assump}
p^{-j}H\equiv g \pmod{p^\ell},
\end{equation}
and assume by contradiction that $j>-\ell$. Note that equation \eqref{eq:assump} implies that $\nu_p(g)=0$. In particular, 
$\nu_p(f_{\frac{A}{B},\epsilon,\wt}+m-p^j g)=\nu_p(H-p^j g) \geq j+\ell$, i.e., $f_{\frac{A}{B},\epsilon,\wt}\equiv p^j g - m \pmod{p^{j+\ell}}$. To conclude the proof it is enough to use Theorem \ref{thm2}.
\end{enumerate}
\end{proof}

\bibliographystyle{elsarticle-num}

\begin{thebibliography}{00}



\bibitem{A} Ahlgren S.,{\it The partition function modulo composite integers m}, Math. Ann., 318:793Ð803, 2000.


\bibitem{AO1}  Ahlgren S. and  Ono K., \textit{Congruence Properties for the Partition Function}, Proc. Natl. Acad. Sci. USA {\bf 98} (2001), no. 23, 12882--12884.
\bibitem{Alf}  Alfes C., \textit{Parity of the coefficients of Klein's $j$-function}, Proc. Amer. Math. Soc., {\bf 14} (2013), 123--130.

\bibitem{Alf2} Alfes C., {\it Congruences for Ramanujan's $\omega(q)$}. Ramanujan J., {\bf 22} (2010), 163--169.

\bibitem{ARZ} Andrews G., Rhoades R., Zwegers S., {\it Modularity of the concave composition generating function}, Accepted for publication in Algebra and Number Theory.

\bibitem{At1} Atkin A. O. L., {\it Proof of a conjecture of Ramanujan}, Glasgow Math. J., {\bf 8} (1967), 14--32.

\bibitem{At2} Atkin A. O. L., {\it Multiplicative congruence properties and density problems for p(n)}, Proc. London Math.
Soc., {\bf 18} (1968) 563--576.

\bibitem{And} Andersen N., {\it Classification of congruences for mock theta functions and weakly holomorphic modular forms}

\bibitem{BO} Bringmann K. and Ono K., {\it Dyson's ranks and Maass forms}, Annals of Mathematics {\bf 171} (2010), 419--449.

\bibitem{BFu} Bruinier J. and Funcke J., \textit{On two geometric theta lifts}, Duke Math. J., {\bf 125} (2004), 45--90. 

\bibitem{BrO}  Bruinier J. and  Ono K., {\it Identities and congruences for Ramanujan-s $\omega(q)$}. Ramanujan J., {\bf 23} (2010), 151--157.

\bibitem{GZ}  Gross B., Zagier D.B., {\it Heegner points and derivatives of L-series}, Invent. Math., {\bf 84} (1986), 225--320.

\bibitem{IRR}  Imamoglu O., Raum M., and  Richter O., \textit{Holomorphic projections and Ramanujan's mock theta functions}, Preprint.



\bibitem{KZ} Kaneko M. and Zagier D., \textit{A generalized Jacobi theta function and quasimodular forms}, in The Moduli Spaces of Curves (R. Dijkgraaf, C. Faber, G. v.d. Geer, eds.), Prog. in Math. {\bf 129}, Birkh\"{a}user, Boston (1995), 165--172.

\bibitem{O3} Ono K.,\textit{Distribution of the partition function modulo $m$},
Annals of Mathematics {\bf 151}, (2000),  293--307.

\bibitem{O2} Ono K., \textit{The web of modularity: Arithmetic of the coefficients of modular forms and q-series}, CBMS Regional Conference Series in Mathematics, {\bf 102}, AMS, Providence, PI (2004).

\bibitem{O1} Ono K., Unearthing the visions of a master: harmonic Maass forms and number theory, {\it Proceedings of the 2008 Harvard-MIT Current Developments in Mathematics Conference}, International Press, Somerville, MA, (2009), 347--454.

\bibitem{R}  Ramanujan S.,{\it Congruence properties of partitions}, Math. Z.,  {\bf 9} (1921), 147--153.


\bibitem{S2} Serre J.P., \textit{Divisibilit\'e des certaines fonctions arithm\'etiques}, Enseign. Math. {\bf 22} (1976), 227--260.

\bibitem{St} Sturm J., \textit{Projections of $C^{\infty}$ automorphic forms}, Bull. AMS {\bf 2} (1980), 435--439.
\bibitem{Tr}  Treneer S., \textit{Congruences for the coefficients of weakly holomorphic modular forms}, Proc. London Math. Soc. {\bf 93} (2006), 304--324.

\bibitem{W} Watson G. N., {\it Ramanujan's vermutung {\"u}ber zerf{\"a}llungsanzahlen}, J. reine Angew. Math., {\bf 179} (1938), 97--128.

\bibitem{Za} Zagier D., {\it Ramanujan's mock theta functions and their applications (after Zwegers and Ono- Bringmann)}
S\'eminaire Bourbaki, 60\'eme ann\`ee, 2007-2008, no 986, Ast\'erisque {\bf 326} (2009), Soc. Math. de France, 143--164

\bibitem{Zw}  Zwegers S., \begin{it}Mock theta functions\end{it}, Ph.D. Thesis, Universiteit Utrecht, (2002).









 \end{thebibliography}



\end{document}